\documentclass[a4paper,12pt]{amsart}

\usepackage{amsmath,amsfonts,amsthm,amssymb,ascmac,amscd}
\usepackage{bm}
\usepackage{graphicx}
\usepackage{enumerate}
\usepackage{overpic}
\usepackage{mathtools}
\usepackage{xcolor}
\numberwithin{equation}{section}
\theoremstyle{definition}
\newtheorem{theorem}{Theorem}[section]
\newtheorem*{theorem*}{Theorem}
\newtheorem{definition}[theorem]{Definition}
\newtheorem*{definition*}{Definition}
\newtheorem{proposition}[theorem]{Proposition}
\newtheorem*{proposition*}{Proposition}

\newtheorem*{example*}{Example}
\newtheorem{remark}[theorem]{Remark}
\newtheorem*{remark*}{Remark}

\newtheorem*{recall*}{Recall}
\newtheorem{lemma}[theorem]{Lemma}
\newtheorem*{lemma*}{Lemma}
\newtheorem{corollary}[theorem]{Corollary}
\newtheorem*{corollary*}{Corollary}

\newtheorem*{question*}{Question}

\newtheorem*{conjecture*}{Conjecture}

\newtheorem*{exercise*}{Exercise}
\newtheorem{claim}[theorem]{Claim}
\newtheorem*{claim*}{Claim}

\newtheorem*{fact*}{Fact}
\newtheorem{theorema}{Theorem}
\newtheorem*{theorema*}{Theorem}

\newtheorem{corollarya}[theorema]{Corollary}
\newtheorem*{corollarya*}{Corollary}
\usepackage{hyperref}
\hypersetup{
bookmarksnumbered=true, 
colorlinks=true, 
citecolor=black,
linkcolor=black,
urlcolor=black,
}
\usepackage{cleveref}
\crefname{def}{Definition}{Definitions}
\crefname{theorem}{Theorem}{Theorems}
\crefname{lemma}{Lemma}{Lemmas}
\crefname{corollary}{Corollary}{Corollaries}
\crefname{proposition}{Proposition}{Propositions}
\crefname{example}{Example}{Examples}
\crefname{remark}{Remark}{Remarks}
\crefname{question}{Question}{Questions}
\crefname{equation}{}{}
\crefname{figure}{Figure}{Figures}
\crefname{theorema}{Theorem}{Theorems}
\crefname{corollarya}{Corollary}{Corollaries}
\newcommand{\Z}{\mathbb{Z}}
\newcommand{\R}{\mathbb{R}}

\newcommand{\multicurve}[1]{\mathcal{C}^{[#1]}}
\newcommand{\interpolate}[1]{\mathcal{P}^{[#1]}}
\DeclareMathOperator{\Mod}{Mod}

\allowdisplaybreaks[4]
\title[large-scale geometry of interpolating graphs]{Large-scale geometry of graphs interpolating between curve graphs and pants graphs}
\author{Erika Kuno, Rin Kuramochi, Kento Sakai}

\address{
(Erika Kuno)
College of Engineering,
Shibaura Institute of Technology,
3-7-5 Toyosu, Koto-ku, Tokyo 135-8548, Japan
}
\email{e-kuno@shibaura-it.ac.jp}

\address{
(Rin Kuramochi)
Graduate School of Mathematical Sciences,
The University of Tokyo,
3-8-1, Komaba, Meguro-ku, Tokyo, 153-8914, Japan
}
\email{kuramochi-rin0311@g.ecc.u-tokyo.ac.jp}

\address{
(Kento Sakai)
Graduate School of Mathematical Sciences,
The University of Tokyo,
3-8-1, Komaba, Meguro-ku, Tokyo, 153-8914, Japan
}
\email{kento@ms.u-tokyo.ac.jp}
\subjclass[2020]{57K20, 20F65, 20F67}
\date{\today}
\begin{document}

\begin{abstract}
    We study two types of graphs interpolating between the curve graph and the pants graph from the viewpoint of large-scale geometry.
    One was introduced by Erlandsson and Fanoni, and the other by Mahan Mj. These graphs were developed independently in different contexts.
    In this paper, we provide explicit formulae for computing their quasi-flat ranks.
    These formulae depend on the genus and the number of boundary components of the underlying surface, as well as the interpolation parameter.
    We also classify geometries of the interpolating graphs into the hyperbolic, relatively hyperbolic, and thick cases.
    Our approach relies on the theory of twist-free graphs of multicurves, which is developed by Vokes and Russel.
\end{abstract}
\maketitle

\section{Introduction}
Let $\Sigma=\Sigma_{g,b}$ be a connected, compact, and orientable surface of genus $g$ with $b$ boundary components.
In this paper, we simply refer to the isotopy class of an essential simple closed curve as a \textit{curve}.
A \textit{pants decomposition} of $\Sigma$ is a multicurve whose complementary components are all pairs of pants (so the multicurve has $3g-3+b$ components).
In the context of mapping class groups, 3-manifolds, and Teichmüller theory, graphs whose vertices correspond to curves or multicurves on $\Sigma$ are extensively studied in terms of large-scale geometry.

Harvey \cite{harvey1981boundary} defined the \textit{curve graph} $\mathcal{C}(\Sigma)$ of $\Sigma$, which is the graph whose vertices are curves on $\Sigma$.
Two distinct vertices of $\mathcal{C}(\Sigma)$ are joined by an edge if the two curves corresponding to the vertices are disjoint on $\Sigma$.
We endow each edge of the graph with length one.
Masur and Minsky \cite{masur1999complex} firstly proved that the curve graph is Gromov hyperbolic.
After the original proof, various other proofs have appeared (for example, see \cite{bowditch2006intersection,hamenstadt2007curves,aougab2013hyperbolicity,clay2014hyperbolicity,bowditch2014hyperbolicity,hensel2015hyperbolicity}, also \cite{kuno2016hyperbolicity} for the case of non-orientable surfaces).

As another graph related to topological surfaces, Hatcher and Thurston \cite{hatcher1980presentation} introduced the pants graph $\mathcal{P}(\Sigma)$ of $\Sigma$. 
Here the \textit{pants graph} of $\Sigma$ is a graph whose
vertices are pants decompositions of $\Sigma$.
Two vertices of $\mathcal{P}(\Sigma)$ are joined by an edge if the corresponding pants decompositions differ by an elementary move, i.e., replacing exactly one curve by another curve with minimal intersection in the complementary subsurface.
We equip $\mathcal{P}(\Sigma)$ with a combinatorial metric in which each edge has length 1.
Brock \cite{brock2003weil} proved that the pants graph is quasi-isometric to the Teichmüller space $\mathcal{T}^{\mathrm{WP}}(\Sigma)$ with the Weil-Petersson metric.
In \cite{behrstock2008dimension}, it is proved that the maximal dimension of the Euclidean space quasi-isometrically embedded into $\mathcal{T}^{\mathrm{WP}}(\Sigma)$ is $\lfloor \frac{3g-2+b}{2}\rfloor$.
In particular, $\mathcal{P}(\Sigma)$ is not Gromov hyperbolic if the complexity of $\Sigma$ ($=3g-3+b$) is larger than or equal to $3$.
Brock and Farb \cite{brock2006curvature}, Behrstock \cite{behrstock2006asymptotic}, and Aramayona \cite{aramayona2006fivetimes} independently  proved that $\mathcal{P}(\Sigma)$ is Gromov hyperbolic if the complexity of $\Sigma$ is equal to $2$.\vspace{2mm}

In two distinct approaches, Erlandsson and Fanoni \cite{erlandsson2017multicurve}, and Mahan Mj \cite{mj2009interpolating} defined graphs of multicurves that interpolate between the curve graph and the pants graph.

Firstly, we introduce the \textit{$k$-multicurve graph} $\multicurve{k}(\Sigma)$ of $\Sigma$ defined by Erlandsson and Fanoni \cite{erlandsson2017multicurve}, where $k$ varies from $1$ to $3g-3+b$.
The vertices of the $k$-multicurve graph $\multicurve{k}(\Sigma)$ are the multicurves on $\Sigma$ consisting of exactly $k$ components.
Two distinct vertices of $\multicurve{k}(\Sigma)$ are joined by an edge if the multicurves corresponding to the vertices minimally intersect (see \Cref{definition:k-multicurve-graph} for details).

On the other hand, Mahan Mj \cite{mj2009interpolating} defined the \textit{complexity-$\xi$ graph} $\interpolate{\xi}(\Sigma)$ for integers $\xi$ between $-1$ and $3g-4+b$.
The complexity-$\xi$ graph is obtained by adding extra edges to the pants graph $\mathcal{P}(\Sigma)$.
Namely, the vertices of $\interpolate{\xi}(\Sigma)$ are pants decompositions, and two distinct vertices are joined by an edge either
(i) if the corresponding pants decompositions satisfy the adjacency condition in the pants graph, or
(ii) if they agree after removing a subsurface whose complexity is at least $\xi$ (see \Cref{definition:complexity-xi-graph} for details).
From the definition, if $\xi$ is $-1$ or $0$, the condition (ii) cannot be satisfied.
Hence, the complexity-$(-1)$ and complexity-$0$ graphs coincide with the usual pants graph $\mathcal{P}(\Sigma)$.
Furthermore, the complexity-$(3g-4+b)$ graph is quasi-isometric to the curve graph $\mathcal{C}(\Sigma)$ \cite[Remark 1.4.1]{mj2009interpolating}.
Thus the complexity-$\xi$ graphs interpolate between the curve graph and the pants graph in terms of large-scale geometry.

We equip the two interpolating graphs between the curve graph and the pants graph with the combinatorial metric.
The $k$-multicurve graphs and the complexity-$\xi$ graphs are apparently distinct graphs of multicurves.
However, these two graphs of multicurves are equivalent from the viewpoint of large-scale geometry.
Let $\xi_0$ be the complexity of $\Sigma=\Sigma_{g,b}$, i.e., $\xi_0\coloneqq 3g-3+b$.
\begin{theorema}\label{theorem:A}
    For each integer $k$ with $1\leq k\leq \xi_0$, the $k$-multicurve graph $\multicurve{k}(\Sigma)$ and the complexity-$(\xi_0-k)$ graph $\interpolate{\xi_0-k}(\Sigma)$ are quasi-isometric.
\end{theorema}

One quasi-isometry invariant of metric spaces is the \textit{quasi-flat rank}, defined as the maximal dimension of the Euclidean space that can be quasi-isometrically embedded into the metric space.
Since the two interpolating graphs are quasi-isometric, they share all quasi-isometry invariants, including the quasi-flat rank.

\begin{theorema}\label{theorem:B}
    Let $k$ be an integer with $1\leq k\leq \xi_0$.
    Then the quasi-flat rank of the $k$-multicurve graph $\multicurve{k}(\Sigma)$ and the complexity-$(\xi_0-k)$ graph $\interpolate{\xi_0-k}(\Sigma)$ are equal to 
    \begin{equation}\label{eq:quasi-flat-rank-for-k-multicurve-graph}
        m(g,b,k)\coloneqq 
        \begin{cases}
            \min \left\{\left\lfloor \frac{2g-2+b}{a(3g-2+b-k)}\right\rfloor, \left\lfloor \frac{3g-2+b}{3g-1+b-k}\right\rfloor \right\} & (b,k)\neq (0,1) \\
            1 & (b,k)=(0,1),
        \end{cases}
    \end{equation}
    where $a(x)\coloneqq \lceil (2x+1)/3 \rceil$ for $x\in \Z$.
\end{theorema}

The essential part of this theorem relies on the results of \cite{vokes2022hirarchical}.
In the paper \cite{vokes2022hirarchical}, Vokes introduced the notion of a \textit{twist-freeness} for graphs whose vertices correspond to multicurves (see \Cref{definition:twistfree-multicurve-graph}).
She proved that a twist-free graph of multicurves is a hierarchically hyperbolic space.
This result allows us to apply the theory of hierarchically hyperbolic spaces,  developed by Behrstock, Hagen and Sisto \cite{behrstock2017hierarchyI,behrstock2019hierarchyII}, to twist-free graphs of multicurves.
This theory reveals many quasi-isometric properties of hierarchically hyperbolic spaces.
By \cite{vokes2022hirarchical}, the large-scale geometry of a twist-free graph of multicurves $\mathcal{G}(\Sigma)$ is determined by the set of \textit{witnesses} for $\mathcal{G}(\Sigma)$.
A witness for $\mathcal{G}(\Sigma)$ is a subsurface of $\Sigma$ that intersects the multicurve corresponding to every vertex of $\mathcal{G}(\Sigma)$.
By \cite[Corollaries 1.3 and 1.4]{vokes2022hirarchical}, the quasi-flat rank of $\mathcal{G}(\Sigma)$ is bounded from above by the maximum number of pairwise disjoint witnesses for $\mathcal{G}(\Sigma)$.
We prove that the maximum number of pairwise disjoint witnesses for the $k$-multicurve graph $\multicurve{k}(\Sigma)$ is given by $m(g,b,k)$ in \Cref{theorem:B}.
We also show that there exists a quasi-isometric embedding $\Z^{m(g,b,k)} \to \multicurve{k}(\Sigma)$.
It is worth noting that the quasi-flat rank is, in general, difficult to give a bound from above, and our upper bound relies on the results of \cite{vokes2022hirarchical,behrstock2021quasiflat}.
In addition, Mahan Mj \cite{mj2009interpolating} proved that the quasi-flat rank of the complexity-$\xi$ graph is equal to the maximum number of pairwise disjoint subsurfaces whose complexities are at least $\xi+1$.
From this result, we obtain another proof of \Cref{theorem:B} via \Cref{theorem:A}.
Note that we cannot directly apply the results of Vokes \cite{vokes2022hirarchical} to the complexity-$\xi$ graph, since it does not satisfy the condition of a twist-free graph of multicurves.\vspace{2mm}

The papers \cite{vokes2022hirarchical,russel2022thickness} provide a classification of twist-free multicurve graphs into the hyperbolic, relatively hyperbolic, and thick cases.
Using this result, we obtain a classification for $k$-multicurve graphs.

\begin{corollarya}\label{corollary:criteria-hyperbolicity}
    Let $\Sigma=\Sigma_{g,b}$ be a connected, compact, and orientable surface of genus $g$ with $b$ boundary components, and assume that $3g-3+b$ is at least $2$.
    Let $k$ be an integer with $1\leq k\leq 3g-3+b$.
    Then the follwing conditions determine whether the $k$-multicurve graph $\multicurve{k}(\Sigma)$ is hyperbolic, relatively hyperbolic, or thick:
    \begin{enumerate}
        \item \textbf{Hyperbolic case.} 
        The $k$-multicurve graph $\multicurve{k}(\Sigma_{g,b})$ is hyperbolic if and only if $m(g,b,k)=1$.
        \item \textbf{Relatively hyperbolic case.} The $k$-multicurve graph $\multicurve{k}(\Sigma_{g,b})$ is relatively hyperbolic if and only if 
        \begin{itemize}
            \item $g$ is even, $b$ is even at least $2$, and $k=(3g+b)/2$,
            \item $g$ is even, $b=0$, and $k\in\{3g/2,(3g+2)/2\}$,
            \item $g$ is odd, $b\in\{0,2\}$, and $k=(3g+3)/2$, or 
            \item $g$ is odd, $b$ is odd at least $3$, and $k=(3g+b)/2$.
        \end{itemize}
        \item \textbf{Thick case.} The $k$-multicurve graph $\multicurve{k}(\Sigma_{g,b})$ is thick if and only if neither (1) nor (2) holds.
    \end{enumerate} 
\end{corollarya}
Using \Cref{theorem:A} and Corollary \ref{corollary:criteria-hyperbolicity}, we verify Conjecture 1 of \cite{mj2009interpolating}.
Mahan Mj's motivation for introducing the complexity-$\xi$ graph is that it provides a quasi-isometric model for the coned-off Cayley graph $\Gamma(\Sigma,\xi)$ of the mapping class group $\Mod(\Sigma)$, where the coning is taken over left cosets of mapping class subgroups associated to subsurfaces of complexity at most $\xi$ \cite{mj2009interpolating}.
Consequently, the $(\xi_0-\xi)$-multicurve graph also provides a quasi-isometric model for this coned-off Cayley graph.

As for the related work, Hamenst\"{a}dt \cite{hamenstadt2014nonseparating} introduced the \textit{nonseparating $k$-multicurve graph $\mathcal{NC}^{[k]}(\Sigma)$} and showed that $\mathcal{NC}^{[k]}(\Sigma)$ is hyperbolic when $k<g/2+1$.
Furthermore, Russel and Vokes \cite{russel2022thickness} classified the large-scale geometry of $\mathcal{NC}^{[k]}(\Sigma)$ into the hyperbolic, relatively hyperbolic, and thick cases.

\subsection*{Acknowledgement}
The authors would like to thank Hidetoshi Masai. 
This work started from his one question.
The third author is grateful to Mahan Mj for answering questions on interpolating graphs.
The second and third authors also would like to thank Nariya Kawazumi for valuable discussions and encouragement.
The first author is supported by JSPS KAKENHI Grant Number 21K13791.
The third author is supported by JSPS KAKENHI Grant Number	25KJ0069.

\section{Preliminaries}

\subsection{Curves, multicurves, and subsurfaces}

Let $\Sigma=\Sigma_{g,b}$ be a connected, compact, and orientable surface of genus $g$ with $b$ boundary components.
A simple closed curve on $\Sigma$ is \textit{essential} if it is not homotopic to a point or to a boundary component of $\Sigma$.
In this paper, we simply refer to the isotopy class of an essential simple closed curve as a \textit{curve}.
A \textit{multicurve} on $\Sigma$ is a set of pairwise disjoint curves on $\Sigma$.
A multicurve is called a \textit{$k$-multicurve} if the number of components of the multicurve is $k$.
A subsurface $S\subset \Sigma$ is \textit{essential} if 
each boundary component of $S$ is an essential curve on $\Sigma$ or is homotopic to a boundary component of $\Sigma$.
Throughout this paper, a \textit{subsurface} refers to the isotopy class of an subsurface.

For an essential subsurface $S\subset \Sigma$, we define the \textit{complexity} $\xi(S)$ as the maximal number of pairwise disjoint essential curves on $S$.
If $S$ has connected components $S_1,\ldots,S_N$, the complexity of $S$ is equal to $\sum_{i=1}^N \xi(S_i)$.
By an Euler characteristic argument, we find $\xi(\Sigma)=3g-3+b$ and a $\xi(\Sigma)$-multicurve gives a pants decomposition of $\Sigma$.
We often refer to a $\xi(\Sigma)$-multicurve as a \textit{pants decomposition}. \vspace{2mm}

Let $S\subset \Sigma$ be an essential subsurface with positive complexity.
We here recall the definition of the \textit{subsurface projection} $\pi_{S}\colon \multicurve{k}(\Sigma) \to 2^{\mathcal{C}(S)}$, which is introduced in \cite{masur1999complex,masur2000complex} (see also \cite{vokes2022hirarchical}).
First, we define $\pi_{S}(\gamma) \in 2^{\mathcal{C}(S)}$ for a single curve $\gamma \in \mathcal{C}(\Sigma)$.
We suppose that $\gamma$ is in minimal position with $\partial S$, i.e., the number of connected components $\gamma\cap S$ and the intersection number $i(\gamma,\partial S)$ are respectively minimized.
If $\gamma \subset S$, then we define $\pi_{S}(\gamma)=\{\gamma\}$.
If $\gamma$ is disjoint from $S$, then we define $\pi_{S}(\gamma)=\varnothing$.
Otherwise, we define $\pi_S(\gamma)\in 2^{\mathcal{C}(S)}$ to be the union, over all arc components $c$ of $\gamma\cap S$, of the set of boundary curves of a closed regular neighborhood of $c\cup \partial S$.
For a $k$-multicurve $\alpha=\{\alpha^1,\ldots,\alpha^k\}\in\multicurve{k}(\Sigma)$, we define $\pi_S(\alpha) \in 2^{\mathcal{C}(S)}$ to be the union of $\pi_S(\alpha^i)$ over $1\leq i\leq k$.

\subsection{Metric geometry}

Let $(X,d_X), (Y,d_Y)$ be metric spaces.
For $K\geq 1$ and $L\geq0$, a map $f\colon X\to Y$ is a \textit{$(K,L)$-quasi-isometric embedding} if for any $x_1,x_2\in X$,
\begin{equation*}
    \frac1{K} d_X(x_1,x_2) -L \leq d_Y(f(x_1),f(x_2)) \leq Kd_X(x_1,x_2) +L.
\end{equation*}
The map $f\colon X\to Y$ is a \textit{quasi-isometric embedding} if it is a $(K,L)$-quasi-isometric embedding for some $K\geq 1,L>0$.
A map $f\colon X\to Y$ is \textit{$C$-quasi-dense} for $C>0$ if the $C$-neighborhood of $f(X)$ contains all of $Y$.
A quasi-isometric embedding $f\colon X\to Y$ is a \textit{quasi-isometry} if $f$ is $C$-quasi-dense for some $C>0$.
Two metric spaces $X$ and $Y$ are said to be \textit{quasi-isometric} if there exists a quasi-isometry between them.

The \textit{quasi-flat rank} of a metric space $X$ is the maximal dimension $n$ of a \textit{quasi-flat}, i.e., a quasi-isometric embedding $\mathbb{Z}^n\to X$. \vspace{2mm}

Following \cite{farb1998relative,brock2006curvature}, we recall the notion of relative hyperbolicity.
Let $X$ be a geodesic metric space and $\mathcal{H}=\{H_\alpha\}_{\alpha\in A}$ be a family of connected subsets of $X$ indexed by $\alpha\in A$.
For each $\alpha\in A$, we introduce a new point $v_\alpha$ and connect every point of $H_\alpha$ to $v_\alpha$ by an edge of length $1/2$.
Let $\hat{X}$ denote the resulting set, equipped with the path metric induced by these edges.
The distance function is denoted by $d_e$.
The resulting metric space $(\hat{X},d_e)$ is called the \textit{electric metric space} (or the \textit{coned-off metric space}) along $\mathcal{H}$.
If the coned-off metric space $\hat{X}$ is a hyperbolic metric space, the metric space $X$ is said to be \textit{weakly hyperbolic relative to $\mathcal{H}$}. 
Furthermore, if $X$ is weakly hyperbolic relative to $\mathcal{H}$ and the pair $(X,\mathcal{H})$ satisfies the \textit{bounded region penetration property} (see \cite{farb1998relative,brock2006curvature} for the definition), then $X$ is said to be \textit{(strongly) hyperbolic relative to $\mathcal{H}$} (see \cite{brock2008coarse,behrstock2009thick,sisto2012metric}).
For a geodesic metric space $X$, we simply say that $X$ is relatively hyperbolic if there exists a family $\mathcal{H}$ of subsets of $X$ such that $X$ is hyperbolic relative to $\mathcal{H}$.

The notion of \textit{thickness} for metric spaces is introduced by Behrstock, Dru\c{t}u and Mosher \cite{behrstock2009thick} (see also \cite[Definition 2.19]{russel2022thickness}).
Thickness is a geometric obstruction to a metric space being relatively hyperbolic \cite{behrstock2009thick}.

\subsection{Interpolating graphs}

In \cite{erlandsson2017multicurve}, Erlandsson and Fanoni introduce the notion of $k$-multicurve graph and show that the automorphism group of the $k$-multicurve graph is the extended mapping class group.

\begin{definition}\label{definition:k-multicurve-graph}
    For each integer $k$ with $1\leq k\leq 3g-3+b$, we define the \textit{$k$-multicurve graph} $\multicurve{k}(\Sigma)$ as follows:
    \begin{itemize}
        \item The vertices of $\multicurve{k}(\Sigma)$ are all $k$-multicurves on $\Sigma$.
        \item Two vertices $\alpha,\beta \in \multicurve{k}(\Sigma)$ are joined by an edge if $\nu\coloneqq \alpha\cap \beta$ is a $(k-1)$-multicurve on $\Sigma$ and single curves $\alpha\setminus \nu$ and $\beta\setminus \nu$ are disjoint when $k\leq 3g-3+b$, or intersect minimally on $\Sigma\setminus \nu$ when $k=3g-3+b$.
    \end{itemize}
\end{definition}

From this definition, we see that $\multicurve{1}(\Sigma)$ is the curve $\mathcal{C}(\Sigma)$ and $\multicurve{3g-3+b}(\Sigma)$ is the pants graph $\mathcal{P}(\Sigma)$. 
We equip the $k$-multicurve graph with the combinatorial metric $d_{\multicurve{k}}$ given by setting each edge to have length $1$.\vspace{2mm}

On the other hand, Mahan Mj introduced other graphs that interpolate between the curve graph and the pants graph \cite{mj2009interpolating}.

\begin{definition}
    Let $\alpha$ be a pants decomposition on $\Sigma$.
    An essential subsurface $S\subset \Sigma$ is \textit{compatible} with $\alpha$ if each curve in $\partial S$ is in $\alpha$ or homotopic to a boundary of $\Sigma$.
\end{definition}

\begin{definition}\label{definition:complexity-xi-graph}
    For each integer $\xi$ with $-1\leq \xi \leq 3g-4+b$, we define the \textit{complexity-$\xi$ graph} (also called \textit{interpolating graph}) $\interpolate{\xi}(\Sigma)$ as follows:
    \begin{itemize}
        \item The vertices of $\interpolate{\xi}(\Sigma)$ are all pants decompositions of $\Sigma$.
        \item Two pants curves $\alpha,\beta$ are joined by an edge if either
        \begin{enumerate}[(i)]
            \item they are joined by an edge in the pants graph, or
            \item there exists an essential subsurface $S\subset \Sigma$ of complexity at most $\xi$ that is compatible with both $\alpha$ and $\beta$, and on whose complement, $\alpha$ and $\beta$ coincide.
        \end{enumerate}
    \end{itemize}
\end{definition}

We equip the interpolating graph $\interpolate{\xi}(\Sigma)$ with the combinatorial metric $d_{\interpolate{\xi}}$ obtained by assigning length $1$ to each edge.
When $\xi=-1$ or $0$, the condition (ii) implies $\alpha$ and $\beta$ coincide as pants decompositions.
Therefore, $\interpolate{-1}(\Sigma)$ and $\interpolate{0}(\Sigma)$ coincide with the pants graph $\mathcal{P}(\Sigma)$.
On the other hand, when $\xi=\xi_0-1$ ($\xi_0\coloneqq 3g-3+b$), the interpolating graph $\interpolate{\xi_0-1}(\Sigma)$ is apparently different from the curve graph $\mathcal{C}(\Sigma)$.
However, as pointed out in \cite{mj2009interpolating}, the interpolating  graph $\interpolate{\xi_0-1}(\Sigma)$ is quasi-isometric to the curve graph $\mathcal{C}(\Sigma)$.

\subsection{Twist-free multicurve graphs}

Vokes introduced the notion of twist-free multicurve graph \cite{vokes2022hirarchical}.
Let $\Sigma$ be a connected, compact, and orientable surface.
Let $\mathcal{G}(\Sigma)$ be a graph of multicurves on $\Sigma$, i.e., a nonempty graph whose vertices are multicurves on $\Sigma$.
We equip $\mathcal{G}(\Sigma)$ with the combinatorial metric $d_{\mathcal{G}}$ defined by setting the length of each edge to be $1$.
We often abuse notation and use the same symbol to denote both a graph and its vertex set.

An connected essential subsurface $S\subset \Sigma$ is a \textit{witness} for a graph $\mathcal{G}(\Sigma)$ of multicurves if every vertex of $\mathcal{G}(\Sigma)$ has an essential intersection with $S$, i.e., for each $\alpha=\{\alpha^1,\ldots,\alpha^m\}\in\mathcal{G}(\Sigma)$, the union $\bigcup_{j=1}^m c^j$ of any representative $c^j$ of each isotopy class $\alpha^j$ intersects with $S$.

\begin{definition}\label{definition:twistfree-multicurve-graph}
    A graph $\mathcal{G}(\Sigma)$ of multicurves is \textit{twist-free} (or \textit{hierarchical}) if it satisfies the following conditions.
    \begin{enumerate}
        \item The graph $\mathcal{G}(\Sigma)$ is connected.
        \item The action of the mapping class group on the vertices of $\mathcal{G}(\Sigma)$ induces automorphisms of $\mathcal{G}(\Sigma)$.
        \item If two vertices $\alpha,\beta\in \mathcal{G}(\Sigma)$ are joined by an edge, the intersection number $i(\alpha,\beta)$ is bounded by a uniform constant $R$.
        \item The set of witnesses for $\mathcal{G}(\Sigma)$ contains no annuli.
    \end{enumerate}
\end{definition}

In the paper \cite{vokes2022hirarchical}, Vokes proves the following theorem.

\begin{theorem}[{\cite[Theorem 1.1]{vokes2022hirarchical}}]
    Let $\mathcal{G}(\Sigma)$ be a twist-free graph of multicurves associated to $\Sigma$.
    Let $\mathfrak{S}(\mathcal{G}(\Sigma))$ be the set of all disjoint unions of disjoint witnesses.
    Then $\mathcal{G}(\Sigma)$ is a hierarchically hyperbolic space with respect to subsurface projections to the curve graphs of subsurfaces in $\mathfrak{S}(\mathcal{G}(\Sigma))$.
\end{theorem}

For the definition of a hierarchically hyperbolic space, see \cite{behrstock2017hierarchyI,vokes2022hirarchical}.
The hierarchical hyperbolicity of $\mathcal{G}(\Sigma)$ leads to many geometric properties of $\mathcal{G}(\Sigma)$.
We recall some results from \cite{vokes2022hirarchical}.

\begin{theorem}[Distance Formula, {\cite[Corollary 1.2]{vokes2022hirarchical}}]\label{theorem:distance-formula}
    Let $\mathcal{G}(\Sigma)$ be a twist-free graph of multicurves on $\Sigma$ and let $\mathfrak{X}(\mathcal{G}(\Sigma))$ be the set of all connected witnesses for $\mathcal{G}(\Sigma)$.
    Then there exists a constant $C_0>0$ such that, for any $C\geq C_0$, there exist $K_1 \geq 1, K_2>0$ such that, for any vertices $\alpha,\beta \in \mathcal{G}(\Sigma)$, we have 
    \begin{equation*}
        d_{\mathcal{G}}(\alpha,\beta) \asymp_{K_1,K_2} \sum_{S\in\mathfrak{X}} [d_{\mathcal{C}(S)}(\pi_S(\alpha),\pi_S(\beta))]_C.
    \end{equation*}
\end{theorem}
Here, $d_{\mathcal{C}(S)}(\pi_S(\alpha),\pi_S(\beta))=\mathrm{diam}_{\mathcal{C}(S)}(\pi_S(\alpha)\cup \pi_S(\beta))$ and $A\asymp_{K_1,K_2} B$ means 
    \begin{equation*}
        K_1^{-1}B-K_2 \leq A\leq K_1B+K_2
    \end{equation*}
and $[\cdot]_C$ is the cutoff function, namely, $[x]_C\coloneqq x$ if $x\geq C$, and $[x]_C\coloneqq 0$ otherwise.

\begin{theorem}[{\cite[Corollary 1.4]{vokes2022hirarchical}}]\label{theorem:upper-bound-quasiflat-rank}
    Let $\mathcal{G}(\Sigma)$ be a twist-free graph of multicurves on $\Sigma$ and let $\nu$ be the maximum of the number of pairwise disjoint witnesses for $\mathcal{G}(\Sigma)$.
    Then, $\nu$ is equal to the largest integer $n$ satisfying the following condition: there exists $K>0$ such that, for any $R>0$, one can take a $(K,K)$-quasi-isometric embedding $B_R^n\to \mathcal{G}(\Sigma)$, where $B_R^n\subset \R^n$ is a $n$-dimensional Euclidean ball of radius $R$.
\end{theorem}

\begin{theorem}[{\cite[Corollary 1.5]{vokes2022hirarchical}}]\label{theorem:vokes-corollary1.5}
    Let $\mathcal{G}(\Sigma)$ be a twist-free graph of multicurves. 
    If there exists no pair of disjoint witnesses for $\mathcal{G}(\Sigma)$, then $\mathcal{G}(\Sigma)$ is Gromov hyperbolic.
\end{theorem}

In \cite{russel2022thickness}, a classification of twist-free graphs of multicurves into hyperbolic, relatively hyperbolic, or thick cases is given.
A subsurface $S\subset \Sigma$ is said to be \textit{co-connected} if the complement $\Sigma \setminus S$ is connected.

\begin{theorem}[{\cite[Theorem 2.25]{russel2022thickness}}]\label{theorem:classification-hyperbolic-relhyperbolic-thick}
    Let $\mathcal{G}(\Sigma)$ be a twist-free graph of multicurves.
    Then the following hold:
    \begin{enumerate}
        \item \textbf{Hyperbolic case.} The graph $\mathcal{G}(\Sigma)$ is hyperbolic if and only if it admits no pair of witnesses that are disjoint.
        \item \textbf{Relatively hyperbolic case.} The graph $\mathcal{G}(\Sigma)$ is relatively hyperbolic if and only if it admits a pair of  witnesses that are disjoint, and whenever connected witnesses $Z,W \subset \Sigma$ for $\mathcal{G}(\Sigma)$ are disjoint and co-connected, then we have $\Sigma\setminus Z=W$.
    \end{enumerate}
    In all cases other than (1) and (2), the graph $\mathcal{G}(\Sigma)$ is a thick metric space.
\end{theorem}

\section{Two kinds of interpolating graphs}

Let $\Sigma=\Sigma_{g,b}$ be the connected, compact, and orientable surface of genus $g$ with $b$ boundary components, and $\xi_0$ denotes the complexity of $\Sigma$, namely $\xi_0=3g-3+b$.
Mahan Mj provided a quasi-isometry map between the curve graph $\mathcal{C}(\Sigma)$ and the complexity-($\xi_0-1$) graph $\interpolate{\xi_0-1}(\Sigma)$ in \cite[Remark 1.4]{mj2009interpolating}.
We naturally extend this quasi-isometry to one between the $k$-multicurve graph $\multicurve{k}(\Sigma)$ and the complexity-$(\xi_0-k)$ graph $\interpolate{\xi_0-k}(\Sigma)$.
For each $k$-multicurve $\alpha=\{\alpha^1,\ldots,\alpha^k\}\in \multicurve{k}(\Sigma)$, we choose any pants decomposition $\tilde{\alpha}$ extending $\alpha$ (i.e.\ $\alpha\subset\tilde{\alpha}$).
We define the map $I\colon \multicurve{k}(\Sigma)\to \interpolate{\xi_0-k}(\Sigma)$ by $I(\alpha)=\tilde{\alpha}$.

\begin{theorem}\label{theorem:quasi-isometry-interplating-graphs}
    Let $\xi_0=3g-3+b$.
    Fix an integer $k$ with $1\leq k\leq \xi_0-1$.
    Then, for any $\alpha,\beta\in \multicurve{k}(\Sigma)$, we have 
    \begin{equation}\label{eq:lipshitz-inequality}
        C_k^{-1}d_{\multicurve{k}}(\alpha,\beta) -1 \leq d_{\interpolate{\xi_0-k}}(I(\alpha),I(\beta))\leq 2d_{\multicurve{k}}(\alpha,\beta),
    \end{equation}
    where $C_k=\min\{k,\xi_0-k\}$.
    Moreover the map $I$ is $1$-quasi-dense.
    In particular, $\multicurve{k}(\Sigma)$ and $\interpolate{\xi_0-k}(\Sigma)$ are quasi-isometric to each other for $k=1,\ldots,\xi_0$.
\end{theorem}

First, we prepare the following lemma.

\begin{lemma}\label{lemma:difference-for-choice-of-extensions}
    Let $\tilde{\alpha} $ and $ \tilde{\alpha}'$ be any pants decompositions which are extensions of $\alpha=\{\alpha^1,\ldots,\alpha^k\}\in\multicurve{k}(\Sigma)$. Then $d_{\interpolate{\xi_0-k}}(\tilde{\alpha}, \tilde{\alpha}')=1$.
\end{lemma}

\begin{proof}
    Set $S=\Sigma\setminus N\!\left(\bigcup_{i=1}^k \alpha^i\right)$, where $N\!\left(\bigcup_{i=1}^k \alpha^i\right)$ denotes a regular neighborhood of $\bigcup_{i=1}^k \alpha^i$.  
    Then $\tilde{\alpha}$ and $\tilde{\alpha}'$ coincide on the subsurface $S$.  
    Since the complexity of $S$ is $\xi_0-k$, the vertices $\tilde{\alpha}$ and $\tilde{\alpha}'$ are joined by an edge of $\interpolate{\xi_0-k}(\Sigma)$.  
\end{proof}

\begin{proof}[Proof of \Cref{theorem:quasi-isometry-interplating-graphs}]
    
    In the definition of the map $I$, there is a choice of an extension of $\alpha$; however, from \Cref{lemma:difference-for-choice-of-extensions}, this changes the image by a distance of at most 1 in the metric $d_{\interpolate{\xi_0-k}}$.
    In addition, \Cref{lemma:difference-for-choice-of-extensions} implies that the map $I$ is 1-quasi-dense.


    Suppose that $k$-multicurves $\alpha=\{\alpha^1,\ldots,\alpha^k\}$ and $\beta=\{\beta^1,\ldots,\beta^k\}$ are joined by an edge of $\multicurve{k}(\Sigma)$.  
    Then $k-1$ components of $\alpha$ and $\beta$ coincide, and the remaining components are disjoint.  
    By relabeling the components if necessary, we may assume that $\alpha^1$ and $\beta^1$ are disjoint and that $\alpha^i=\beta^i$ for $2\le i \le k$.  
    Since the curves $\beta^1,\alpha^1,\alpha^2,\ldots,\alpha^k$ are pairwise disjoint, we can take a pants decomposition $P$ extending these $k+1$ curves.  
    Both the pants decomposition $P$ and $I(\alpha)$ are extensions of $\alpha$, and similarly both $P$ and $I(\beta)$ are extensions of $\beta$.  
    By \Cref{lemma:difference-for-choice-of-extensions}, we have
    \begin{equation*}
        d_{\interpolate{\xi_0-k}}(I(\alpha),P)
        = d_{\interpolate{\xi_0-k}}(I(\beta),P)
        = 1.
    \end{equation*}
    Thus, for $\alpha,\beta\in\multicurve{k}(\Sigma)$ with $d_{\multicurve{k}}(\alpha,\beta)=1$, we obtain
    \begin{equation*}
        d_{\interpolate{\xi_0-k}}(I(\alpha), I(\beta)) \leq 2.
    \end{equation*}    
    This proves the right-hand inequality in \Cref{eq:lipshitz-inequality}.

    Next, we prove the left-hand inequality.
    We fix two $k$-multicurves $\alpha, \beta\in\multicurve{k}(\Sigma)$ and put $n\coloneqq d_{\interpolate{\xi_0-k}}(I(\alpha),I(\beta))$.
    We take a geodesic $\{P_i\}_{i=0}^n \subset \multicurve{\xi_0-k}(\Sigma)$ between $I(\alpha)$ and $I(\beta)$, where $P_0=I(\alpha), P_n=I(\beta)$ and $d_{\interpolate{\xi_0-k}} (P_{i-1},P_i)=1$ for each $i$ with $1\leq i\leq n-1$.
    We fix an integer $i$ with $2\leq i\leq n$.
    Now, $P_{i-1}$ and $P_{i}$ satisfy either the condition (i) or (ii) in \Cref{definition:complexity-xi-graph}.
    In both cases, the two pants decompositions coincide outside a subsurface $S\subset \Sigma$ whose complexity is $\xi_0-k$.
    In other words, there is a $k$-multicurve $\gamma_i=\{\gamma_i^1,\ldots,\gamma_i^k\}\in \multicurve{k}(\Sigma)$ such that it lies in $\Sigma\setminus S$ and is contained in both $P_{i-1}$ and $P_i$.
    Therefore $P_{i-1}$ contains both $\gamma_{i-1}$ and $\gamma_i$.
    Then, we have 
    \begin{equation*}
        d_{\multicurve{k}}(\gamma_{i-1},\gamma_i)\leq \min\{k,\xi_0-k\}\eqqcolon C_k,
    \end{equation*}
    since $\gamma_i$ is obtained from $\gamma_{i-1}$ by replacing a curve of $\gamma_{i-1}$ with a curve of $\gamma_i$, and the number of such replacements is bounded by $C_k$.
    Moreover, $P_0=I(\alpha)$ contains both $\alpha$ and $\gamma_1$, and $P_n=I(\beta)$ contains both $\beta$ and $\gamma_n$. 
    Therefore, we obtain 
    \begin{align*}
        d_{\multicurve{k}}(\alpha,\beta) &\leq d_{\multicurve{k}}(\alpha,\gamma_1)+\sum_{i=2}^n d_{\multicurve{k}}(\gamma_{i-1},\gamma_i)+d_{\multicurve{k}}(\gamma_n,\beta) \\
        &\leq (n+1)C_k.
    \end{align*}
    This implies the left-hand inequality in \Cref{eq:lipshitz-inequality}.
\end{proof}

Based on the result of \cite[Theorem 1.3]{masur1999complex}, Mahan Mj \cite{mj2009interpolating} introduced the complexity-$\xi$ graph $\interpolate{\xi}(\Sigma)$ as a quasi-isometric model of the electrified Cayley graph of the mapping class group $\Mod(\Sigma)$ with respect to subsurfaces whose complexity is less than or equal to $\xi$.

Following \cite{mj2009interpolating}, we recall the construction of electrified Cayley graphs of $\Mod(\Sigma)$.
Fix an integer $\xi$ with $-1\leq \xi\leq 3g-4+b$.
Let $\mathfrak{S}_\xi(\Sigma)$ be the set of all essential subsurfaces of $\Sigma$ whose complexity is at least $\xi$.
Let $S_1,\ldots,S_k$ be subsurfaces constituting a complete set of representatives for the quotient $\Mod(\Sigma)\backslash \mathfrak{S}_\xi(\Sigma)$.
Choose a finite generating set of $\Mod(\Sigma)$, and let $\Gamma(\Sigma)$ be the corresponding Cayley graph of $\Mod(\Sigma)$.
We cone off $\Gamma(\Sigma)$ along all left cosets of subgroups $\Mod(S_i)$ for $i=1,\ldots,k$.
The resulting metric space is denoted by $\Gamma(\Sigma,\xi)$.

Let $\xi_0=3g-3+b$.
Masur and Minsky \cite{masur1999complex} showed that the curve graph $\mathcal{C}(\Sigma)$ is quasi-isometric to $\Gamma(\Sigma,\xi_0-1)$, and Mahan Mj \cite{mj2009interpolating} showed that $\interpolate{\xi}(\Sigma)$ is quasi-isometric to $\Gamma(\Sigma,\xi)$ for each $\xi$ with $-1\leq \xi \leq \xi_0-1$.
By combining these results with \Cref{theorem:quasi-isometry-interplating-graphs}, we obtain the following summary. 

\begin{corollary}
    The following metric spaces are mutually quasi-isometric:
    \begin{itemize}
        \item the $k$-multicurve graph $\multicurve{k}(\Sigma)$,
        \item the complexity-$(\xi_0-k)$ graph $\interpolate{\xi_0-k}(\Sigma)$, and 
        \item the electrified Cayley graph $\Gamma(\Sigma,\xi_0-k)$.
    \end{itemize}
\end{corollary}

\section{Quasi-isometric geometry of interpolating graphs}

\subsection{Hyperbolicity of two kinds of interpolating graphs}

We fix an integer $k$ with $1\leq k\leq 3g-3+b$.
For the $k$-multicurve graph $\multicurve{k}(\Sigma)$, we obtain the following.

\begin{proposition}\label{proposition:witness-for-k-multicurve}
    For a connected essential subsurface $S\subset \Sigma$, the following conditions are equivalent:
    \begin{enumerate}
        \item The subsurface $S$ is a witness for the $k$-multicurve graph $\multicurve{k}(\Sigma)$.
        \item The complexity $\xi(S)$ is at least $3g-3+b-(k-1)$.
    \end{enumerate} 
\end{proposition}

\begin{proof}
    We suppose that the negation of the condition (1) holds; namely, the essential subsurface $S$ is not a witness for $\multicurve{k}(\Sigma)$.
    Then we can find at least $k$ pairwise disjoint essential curves on $\Sigma$ lying outside $S$.
    Since the total number of pairwise disjoint essential curves on $\Sigma$ is $3g-3+b$, the maximal number of pairwise disjoint essential curves on $S$ is at most $3g-3+b-k$.
    This implies that $\xi(S)$ is at most $3g-3+b-k$.
    Therefore, we see that the negation of condition (1) implies the negation of condition (2).
    The converse direction can  also be verified, and hence we obtain the desired conclusion.
\end{proof}

Erlandsson and Fanoni prove that the $k$-multicurve graph $\multicurve{k}(\Sigma)$ is connected, and hence $\multicurve{k}(\Sigma)$ satisfies condition (1) in \Cref{definition:twistfree-multicurve-graph} \cite[Lemma 2.2]{erlandsson2017multicurve}.
We verify that the $k$-multicurve graph $\multicurve{k}(\Sigma)$ satisfies conditions (2)--(4) in \Cref{definition:twistfree-multicurve-graph}.

\begin{proposition}
    The $k$-multicurve graph $\multicurve{k}(\Sigma)$ is twist-free.
\end{proposition}
Let $m(g,b,k)$ be the maximal number of pairwise disjoint subsurfaces of $\Sigma_{g,b}$ whose complexities are at least $3g-3+b-(k-1)$.
From \Cref{theorem:vokes-corollary1.5}, we obtain the following.

\begin{theorem}
    The $k$-multicurve graph $\multicurve{k}(\Sigma_{g,b})$ is Gromov hyperbolic if and only if $m(g,b,k)=1$.
\end{theorem}

In \Cref{section:maximal-number-of-witnesses}, we provide a formula for computing the number $m(g,b,k)$.

\subsection{Quasi-flat embedding}

Let $m=m(g,b,k)$ be the maximum number of pairwise disjoint subsurfaces on $\Sigma$ whose complexities are at least $3g-3+b-(k-1)$.
In this subsection, following \cite{brock2006curvature} and using the distance formula (\Cref{theorem:distance-formula}), we directly construct a quasi-isometric embedding $\mathbb{Z}^m \to \multicurve{k}(\Sigma)$.

\begin{theorem}\label{theorem:quasi-flat-embedding-rank-nu}
    There exists a quasi-isometric embedding $\mathbb{Z}^m\to \multicurve{k}(\Sigma)$.
\end{theorem}

\begin{remark}
    In \cite{mj2009interpolating}, Mahan Mj determined the quasi-flat rank of the complexity-$\xi$ graph $\interpolate{\xi}(\Sigma)$.
    Combining the result with \Cref{theorem:quasi-isometry-interplating-graphs}, we find the quasi-flat rank of the $k$-multicurve graph, however we instead provide a direct construction of a quasi-isometric embedding into $\multicurve{k}(\Sigma)$, following the construction of \cite{brock2006curvature}.
\end{remark}

\begin{proof}[Proof of \Cref{theorem:quasi-flat-embedding-rank-nu}]
    Let $S_1,\ldots,S_m \subset \Sigma$ be pairwise disjoint subsurfaces of complexity $3g-3+b-(k-1)$.
    Since the case of $k=3g-3+b$ was proven in \cite{brock2006curvature}, we may assume the complexity $\xi(S_i)$ of each $S_i$ is at least $2$. 
    From \Cref{proposition:witness-for-k-multicurve}, the subsurface $S_1$ is not a witness for $\multicurve{k-1}(\Sigma)$, since $\xi(S_1)\leq 3g-3+b-(k-2)$.
    Therefore, we can choose a $(k-1)$-multicurve $\gamma$ such that each component of $\gamma$ lies in $\Sigma\setminus S_1$ and is either disjoint from or parallel to every component of $\partial S_i$ for all $i=1,\ldots,m$.
    Moreover, for each $i$ with $2\leq i\leq m$, we may assume that the subsurface $S_i$ contains at least one component of $\gamma$ as an essential curve in $S_i$.
    We take an additional essential curve $c$ in $S_1$.
    Then $\alpha\coloneqq \{c\}\cup \gamma$ is a $k$-multicurve on $\Sigma$.
    We will construct a quasi-flat $Q\colon \mathbb{Z}^m\to \multicurve{k}(\Sigma)$ with $Q(\bm{0})=\alpha$.
    Let $a_i(0),b_i^1(0),\ldots,b_i^{l_i}(0) \in \alpha $ be the curves that are essentially contained in $S_i$.
    On $S_1$, there is a unique curve $c\in \alpha$, which we denote by $a_1(0)$.
    Let $\gamma_0\subset \alpha$ denote the family of curves lying outside $\bigcup_{i=1}^m S_i$.

    For each $i=1,\ldots, m$, let $a_i\colon \Z\to \mathcal{C}(S_i)$ be a bi-infinite geodesic curve with the initial point $a_i(0)$ for each $i$ with $1\leq i\leq m$.
    For every $k\in \Z$ and an integer $i$ with $2\leq i\leq m$, we take a $l_i$-multicurve 
    \begin{equation*}
        \{b_i^1(k),\ldots,b_i^{l_i}(k)\} \in \multicurve{\ell_i}(S_i)
    \end{equation*}
    such that each $b_i^j(k)\ (1\leq j\leq l_i)$ is disjoint from $a_i(k)$ (possibly $l_i=0$).
    Let $\gamma_i(k)$ denote the $(l_i+1)$-curve $\{a_i(k),b_i^1(k), \ldots, b_i^{l_i}(k)\} \in \mathcal{C}(S_i)$.
    For each $(k_1,\ldots,k_m) \in \Z^m$, we define 
    \begin{align*}
        Q(k_1,\ldots,k_m) =\gamma_0\cup \gamma_1(k_1)\cup \gamma_2(k_2)\cup \cdots\cup \gamma_m(k_m) \in \multicurve{k}(\Sigma).
    \end{align*}
    We prove that the map $Q\colon \Z^m \to \multicurve{k}(\Sigma)$ is a quasi-isometric embedding.

    First, since $Q(k_1,k_2,\ldots, k_m)$ and $Q(k_1\pm 1, k_2,\ldots,k_m)$ are clearly disjoint, the $\multicurve{k}$-distance between them is equal to $1$.
    Therefore we obtain 
    \begin{equation*}
        d_{\multicurve{k}}(Q(k_1,k_2,\ldots, k_m),Q(j_1, k_2,\ldots,k_m))\leq |k_1-j_1|.
    \end{equation*}

    Next, we fix $i$ with $2\leq i\leq m$.
    For each $\bm{k}=(k_1,\ldots,k_m)$, we consider the distance $d_{\multicurve{k}}(Q(\bm{k}),Q(\bm{k}\pm\bm{e}_i))$, where $\bm{e}_i$ is the $i$-th standard basis of $\Z^m$.
    Now, the subsurface $S_1$ can contain $3g-3+b-(k-1)$ essential curves in $S_1$.
    Since $l_i+1$ is at most $3g-3+b-(k-1)$, we can take $l_i$ essential curves $c_1,\ldots,c_{l_i}$ in $S_1$ such that $a_1(k_1),c_1,\ldots,c_{l_i}$ are pairwise disjoint.
    Replacing curves $b_i^1(k_i),\ldots,b_i^{l_i}(k_i)$ in $Q(\bm{k})$ with $c_1,\ldots,c_{l_i}$, we denote the resulting $k$-multicurve by $Q(\bm{k})'$.
    Then, the distance between $Q(\bm{k})$ and $Q(\bm{k})'$ is at most $l_i$.
    In addition, we move $a_i(k_i)$ of $Q(\bm{k})'$ to $a_i(k_i\pm 1)$, and then replace $c_1,\ldots,c_{l_i}$ by $b_i^1(k_i\pm 1),\ldots,b_i^{l_i}(k_i\pm 1)$ (where the sign $\pm$ is taken consistently).
    Then, the replaced multicurve is $Q(\bm{k}\pm \bm{e}_i)$.
    Since $Q(\bm{k})'$ differs from $Q(\bm{k}\pm \bm{e}_i)$ by $l_i+1$ curves,
    the distance $d_{\multicurve{k}}(Q(\bm{k})',Q(\bm{k}\pm \bm{e}_i))$ is at most $l_i+1$.
    Therefore, we have 
    \begin{align*}
        d_{\multicurve{k}}(Q(\bm{k}),Q(\bm{k}\pm\bm{e}_i)) & \leq d_{\multicurve{k}}(Q(\bm{k}),Q(\bm{k})')+d_{\multicurve{k}}(Q(\bm{k})',Q(\bm{k}\pm\bm{e}_i)) \\
        & \leq 2l_i+1 \\
        &\leq 2(3g-3+b-k)+1 \eqqcolon C_{g,b,k}.
    \end{align*}
    Thus, we obtain
    \begin{equation}\label{eq:lipshitz-inequaltiy-for-quasiflat}
        d_{\multicurve{k}}(Q(\bm{k}),Q(\bm{j}))\leq C_{g,b,k} d_{\Z^m}(\bm{k},\bm{j}),
    \end{equation}
    where $d_{\Z^m}(\bm{k},\bm{j})=\sum_{i=1}^m |k_i-j_i|$.

    On the other hand, by \Cref{theorem:distance-formula}, there exist $C_0>0, K_1\geq 1, K_2>0$ such that 
    \begin{equation*}
        K_1 d_{\multicurve{k}}(Q(\bm{k}),Q(\bm{j})) +K_2 
        \geq  \sum_{S\in \mathfrak{X}(\multicurve{k}(\Sigma))} [d_{S}(\pi_S(Q(\bm{k})),\pi_S(Q(\bm{j})))]_{C_0},
    \end{equation*}
    where $d_S(A,B)\coloneqq \mathrm{diam}_{\mathcal{C}(S)}(A\cup B)$ for subsets $A,B\subset \mathcal{C}(\Sigma)$.
    For each $i$ with $1\leq i\leq m$, 
    \begin{align*}
        d_{S_i}(\pi_{S_i}(Q(\bm{k})),\pi_{S_i}(Q(\bm{j}))) &= \mathrm{diam}_{\mathcal{C}(S_i)} (\gamma_i(k_i)\cup \gamma_i(j_i)) \\
        &\geq d_{\mathcal{C}(S_i)}(a_i(k_i),a_i(j_i)) \\
        &=|k_i-j_i|.
    \end{align*}
    Therefore, we obtain
    \begin{align*}
         K_1 d_{\multicurve{k}}(Q(\bm{k}),Q(\bm{j})) +K_2 
        &\geq \max_{1\leq i\leq m} |k_i-j_i| \\
        &\geq \frac{1}{m} \sum_{i=1}^m |k_i-j_i| =\frac1m d_{\Z^m}(\bm{k},\bm{j}).
    \end{align*}
    Together with the inequality \Cref{eq:lipshitz-inequaltiy-for-quasiflat}, this gives the desired conclusion.
\end{proof}

By \Cref{theorem:quasi-flat-embedding-rank-nu}, we find that the number $m(g,b,k)$ provides a lower bound for the quasi-flat rank of $\multicurve{k}(\Sigma)$.
On the other hand, \Cref{theorem:upper-bound-quasiflat-rank} implies that the number $m(g,b,k)$ is also an upper bound for the quasi-flat rank.
Therefore, we have the following corollary.

\begin{corollary}
    Let $m$ be the maximum number of pairwise disjoint subsurfaces on $\Sigma$ whose complexities are at least $3g-3+b-(k-1)$.
    The quasi-flat rank of the $k$-multicurve graph $\multicurve{k}(\Sigma)$ is equal to $m$.
\end{corollary}

\section{Arguments for subsurfaces}

\subsection{Maximal number of witnesses}\label{section:maximal-number-of-witnesses}

Let $\xi$ be a positive integer. 
We will consider a finite collection of simple closed curves on $\Sigma=\Sigma_{g, b}$ along which $\Sigma$ is cut into compact subsurfaces $X_1, \dots, X_m$ such that the complexity $\xi(X_i)$ of each $X_i$ is at least $\xi$. 
We denote by $\mu(g, b, \xi)$ the maximal possible number of subsurfaces appearing in such decompositions. 
In this subsection, we explicitly calculate $\mu(g, b, \xi)$ as follows. 
To the best of the authors' knowledge, no such explicit formula has been given so far.

\begin{theorem}\label{theorem:MaximumNumberForSubsurfaces}
    Unless $b=0$ and $\xi=3g-3$, $\mu(g, b, \xi)$ is equal to 
    \begin{equation}\label{eq:ComplexityFormula}
        \min\left\{\left\lfloor\frac{3g-2+b}{\xi+1}\right\rfloor,\ \left\lfloor\frac{2g-2+b}{\lceil (2\xi+1)/3\rceil}\right\rfloor\right\}. 
    \end{equation}
    For $g\ge 2$, $\mu(g, 0, 3g-3)$ is equal to $1$. 
\end{theorem}

\begin{proof}
    Since the claim is trivial for $\xi(\Sigma)<\xi$, we assume $\xi(\Sigma)\ge \xi$ below. In particular $\mu(g, b, \xi)$ is at least one. 

    By cutting along some simple closed curves on $\Sigma$, we obtain compact subsurfaces $X_1, \dots, X_{\mu(g, b, \xi)}$ whose complexities are at least $\xi$. Denote the genus and the number of boundary components of $X_i$ by $g_i$ and $b_i$, respectively. Since the Euler characteristic of $X_i$ satisfies $\chi(X_i)=2-2g_i-b_i=-\xi(X_i)-1+g_i\le -\xi-1+g_i$, we have
    \begin{align*}
        2-2g-b
        &=\sum_{i=1}^{\mu(g, b, \xi)} \chi(X_i)\\
        &\le -(\xi+1)\mu(g, b, \xi)+\sum_{i=1}^{\mu(g, b, \xi)} g_i\\
        &\le -(\xi+1)\mu(g, b, \xi)+g. 
    \end{align*}
    Therefore, $\mu(g, b, \xi)$ is at most $\lfloor(3g-2+b)/(\xi+1)\rfloor$. 

    On the other hand, each $b_i$ is positive unless $b=0$ and $\mu(g, 0, \xi)\le 1$. Using $b_i\ge 1$ together with $3g_i-3+b_i=\xi(X_i)\ge \xi$ and $\chi(X_i)=2-2g_i-b_i$, we obtain
    \begin{equation*}
        \xi+1+\chi(X_i)\le g_i \le \frac{-\chi(X_i)+1}{2}. 
    \end{equation*}
     In particular, the rightmost side is greater than or equal to the leftmost side, i.e., $\chi(X_i)\le -(2\xi+1)/3$. Since $\chi(X_i)$ is an integer, $\chi(X_i)\le-\lceil (2\xi+1)/3\rceil$. Thus it follows that
    \begin{equation*}
        -\mu(g, b, \xi)\left\lceil \frac{2\xi+1}{3}\right\rceil\ge \sum_{i=1}^{\mu(g, b, \xi)} \chi(X_i)=2-2g-b, 
    \end{equation*}
    i.e., $\mu(g, b, \xi)$ is at most $\lfloor(2g-2+b)/\lceil (2\xi+1)/3\rceil\rfloor$. 

    In the case of $b=0$ and $\xi<3g-3$, let $r\in \{0, 1, 2\}$ be the integer such that $\xi \equiv r \mod 3$. Then $\xi+3-r$ is at most $3g-3$, and we have 
    \begin{align*}
        \left\lfloor\frac{2g-2}{\lceil (2\xi+1)/3\rceil}\right\rfloor
        &= \left\lfloor\frac{2(3g-3)}{3\lceil (2(\xi-r)+2r+1)/3\rceil}\right\rfloor \\
        &\ge \left\lfloor\frac{2(\xi-r+3)}{2(\xi-r)+3\lceil (2r+1)/3\rceil}\right\rfloor \\
        &\ge 1. 
    \end{align*}
    Thus \cref{eq:ComplexityFormula} is at least one. 
    
    The arguments so far prove that $\mu(g, b, \xi)$ admits an upper bound given by \cref{eq:ComplexityFormula} unless $b=0$ and $\xi=3g-3$. We need to see that the quantity \cref{eq:ComplexityFormula} is also a lower bound for $\mu(g, b, \xi)$. 
    
    First, note that there exists a separating simple closed curve along which the surface $\Sigma$ is cut into two subsurfaces
    \begin{equation}\label{eq:LowerGenusDecomp}
        X_0=\Sigma_{g, \xi-3g+3} \quad \text{and} \quad Y_0=\Sigma_{0, b-\xi+3g-1}
    \end{equation}
    if $3g-3<\xi$. The complexity of $X_0$ is equal to $\xi$. In addition, we note that $\lfloor (3g-2+b)/(\xi+1) \rfloor$ can be rewritten as $\lfloor (\xi(\Sigma)+1)/(\xi+1)\rfloor$. 
    
    \textbf{Case 1}. 
    Suppose $\xi\equiv 1\mod 3$. We will prove $\mu(g, b, \xi)\ge\lfloor(3g-2+b)/(\xi+1)\rfloor$ by induction on $\xi(\Sigma)=3g-3+b$. There is a simple closed curve on $\Sigma$ which cuts $\Sigma$ into $X_1$ and $Y_1$ with $\xi(X_1)=\xi$. Indeed, one can take $X_1=\Sigma_{(\xi+2)/3, 1}$ and $Y_1=\Sigma_{g-(\xi+2)/3, b+1}$ when $g\ge (\xi+2)/3$ as in \Cref{fig:mod1decomp}, whereas we set $X_1=X_0$ and $Y_1=Y_0$ in the decomposition \cref{eq:LowerGenusDecomp} when $g<(\xi+2)/3$. The complexity of $Y_1$ is equal to $\xi(\Sigma)-\xi-1$ in either case. 
    \begin{figure}[t]
        \vspace{3mm}
        \centering
        \begin{overpic}[scale=0.8]{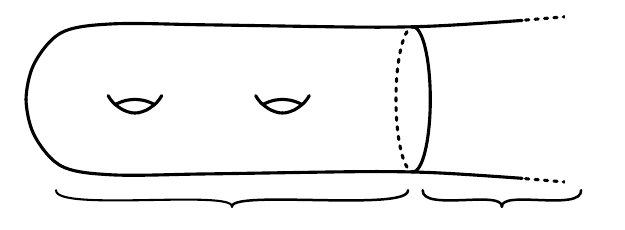}
            \put(23,-1){$X_1=\Sigma_{(\xi+2)/3,1}$}
            \put(65,-1){$Y_1=\Sigma_{g-(\xi+2)/3, b+1}$}
        \end{overpic}
        \caption{}
        \label{fig:mod1decomp}
    \end{figure}
    By the induction hypothesis for $Y_1$, we obtain
    \begin{align}\label{eq:ComplexityInduction}
        \begin{split}
            \mu(g, b, \xi)
            &\ge 1+\left\lfloor\frac{(\xi(\Sigma)-\xi-1)+1}{\xi+1}\right\rfloor\\
            &=\left\lfloor\frac{\xi(\Sigma)+1}{\xi+1}\right\rfloor
            =\left\lfloor\frac{3g-2+b}{\xi+1}\right\rfloor. 
        \end{split}
    \end{align}

    \textbf{Case 2}. Suppose $\xi\equiv 2 \mod 3$. First, we show $\mu(g, b, \xi)\ge \lfloor(3g-2+b)/(\xi+1)\rfloor$ for $b>0$, by induction on $\xi(\Sigma)$. We can take a simple closed curve on $\Sigma$ so that the surface $\Sigma$ is cut into two subsurfaces $X_2$ and $Y_2$ with $\xi(X_2)=\xi$. Indeed, one can take $X_2=\Sigma_{(\xi+1)/3, 2}$ and $Y_2=\Sigma_{g-(\xi+1)/3, b}$ when $g\ge (\xi+1)/3$ as in \Cref{fig:mod2decomp}, whereas we set $X_2=X_0$ and $Y_2=Y_0$ in the decomposition \cref{eq:LowerGenusDecomp} when $g<(\xi+1)/3$. The complexity of $Y_2$ is equal to $\xi(\Sigma)-\xi-1$ in either case. 
    \begin{figure}[t]
        \vspace{3mm}
        \centering
        \begin{overpic}[scale=0.8]{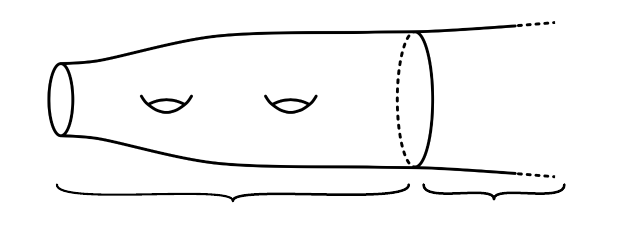}
            \put(23,0){$X_2=\Sigma_{(\xi+1)/3, 2}$}
            \put(65,0){$Y_2=\Sigma_{g-(\xi+1)/3, b}$}
        \end{overpic}
        \caption{}
        \label{fig:mod2decomp}
    \end{figure}
    Hence, the same inequality as \cref{eq:ComplexityInduction} holds by the induction hypothesis for $Y_2$, and we have $\mu(g, b, \xi)\ge \lfloor(3g-2+b)/(\xi+1)\rfloor$. 
    
    If $b=0$, we obtain $\Sigma_{g-1, 2}$ by cutting $\Sigma$ along a non-separating simple closed curve. Therefore
    \begin{align*}
        \mu(g, 0, \xi)
        &\ge \mu\left(g-1, 2, \xi\right)\\
        &\ge \left\lfloor\frac{3(g-1)-2+2}{\xi+1}\right\rfloor\\
        &=\left\lfloor\frac{3g-3}{\xi+1}\right\rfloor
        =\left\lfloor\frac{2g-2}{\lceil (2\xi+1)/3\rceil}\right\rfloor. 
    \end{align*}

    \textbf{Case 3}. Suppose $\xi\equiv 0\mod 3$. If $b\ge 2$, we can take a simple closed curve on $\Sigma$ which cuts $\Sigma$ into two subsurfaces
    \begin{equation}\label{eq:Mod3Decomp-i}
        X_3 \quad \text{and} \quad Y_3
    \end{equation}
    with $\xi(X_3)=\xi$. Indeed, one can take $X_3=\Sigma_{\xi/3, 3}$ and $Y_3=\Sigma_{g-\xi/3, b-1}$ when $g\ge \xi/3$ as in the left figure of \Cref{fig:mod3decomp}; whereas we set $X_3=X_0$ and $Y_3=Y_0$ in the decomposition \cref{eq:LowerGenusDecomp} when $g<\xi/3$. The complexity of $Y_3$ is equal to $\xi(\Sigma)-\xi-1$ in either case. 
    
    For $b=1$, one can take two simple closed curves that cut $\Sigma$ into
    \begin{equation}\label{eq:Mod3Decomp-ii}
        X_3'=\Sigma_{\xi/3, 3} \quad \text{and} \quad Y_3'=\Sigma_{g-1-\xi/3, 2}
    \end{equation}
    as in the right figure of \Cref{fig:mod3decomp}. 
    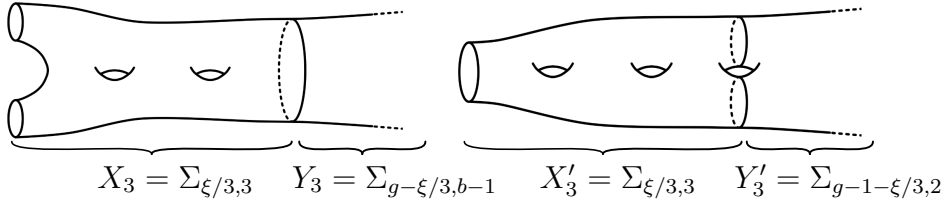
\begin{figure}[t]
        \centering
        \begin{overpic}[scale=0.6]{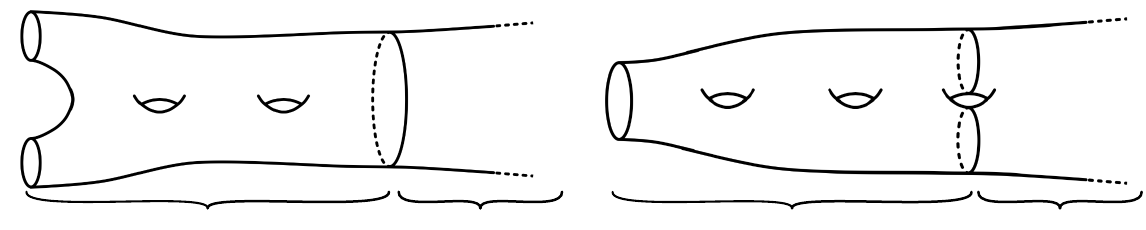}
            \put(12,-2){$X_3=\Sigma_{\xi/3,3}$}
            \put(34,-2){$Y_3=\Sigma_{g-\xi/3, b-1}$}
            \put(62,-2){$X_3'=\Sigma_{\xi/3, 3}$}
            \put(84,-2){$Y_3'=\Sigma_{g-1-\xi/3, 2}$}
        \end{overpic}
        \vspace{5mm}
        \caption{The left and right figures represent the decompositions \Cref{eq:Mod3Decomp-i} and \Cref{eq:Mod3Decomp-ii}, respectively.}
        \label{fig:mod3decomp}
    \end{figure}
    
    (3-I) We will show $\mu(g, b, \xi)\ge \lfloor(3g-2+b)/(\xi+1)\rfloor$ for $b>\lfloor 3g/\xi \rfloor$ by induction on $\xi(\Sigma)$. In this case, $b$ is at least $2$ due to $b>\lfloor 3g/\xi\rfloor$ and $\xi(\Sigma)\ge \xi$, and hence the decomposition \cref{eq:Mod3Decomp-i} holds. Denote the genus and the number of boundary components of $Y_3$ by $g'$ and $b'$, respectively. Then the subsurface $Y_3$ still satisfies $b'>\lfloor 3g'/\xi\rfloor$ by $b-1>\lfloor 3(g-\xi/3)/\xi\rfloor$. Thus, it follows that $\mu(g, b, \xi)\ge \lfloor(3g-2+b)/(\xi+1)\rfloor$ by the same calculation as \cref{eq:ComplexityInduction}. 

    (3-II) Next, we consider the case of $0<b\le \lfloor 3g/\xi\rfloor$ and show $\mu(g, b, \xi)=\lfloor(2g-2)/\lceil (2\xi+1)/3\rceil\rfloor$ by induction on $\xi(\Sigma)$. The situation is further divided into three cases according to the number of boundary components. 
    
    (3-II-i) For $2\le b\le \lfloor(2g-2)/\lceil (2\xi+1)/3\rceil\rfloor$, using the decomposition \cref{eq:Mod3Decomp-i}, we obtain
    \begin{align*}
        \mu(g, b, \xi)
        &\ge 1+\mu\left(g-\frac{\xi}{3}, b-1, \xi\right)\\
        &\ge 1+\left\lfloor\frac{2\left(g-\frac{\xi}{3}\right)-2+(b-1)}{\lceil (2\xi+1)/3\rceil}\right\rfloor\\
        &=\left\lfloor\frac{2g-2+b}{\lceil (2\xi+1)/3\rceil}\right\rfloor. 
    \end{align*}
    
    (3-II-ii) Consider the case of $b=1$ and $2\le \lfloor 3(g-1-\xi/3)/\xi\rfloor$. Then we have the decomposition \cref{eq:Mod3Decomp-ii}. The assumption $2\le \lfloor 3(g-1-\xi/3)/\xi\rfloor$ enables us to apply the induction hypothesis to $Y_3'$, and we obtain
    \begin{align*}
        \mu(g, 1, \xi)
        &\ge 1+\mu\left(g-\frac{\xi}{3}-1, 2, \xi\right)\\
        &\ge 1+\left\lfloor\frac{2\left(g-\frac{\xi}{3}-1\right)-2+2}{\lceil (2\xi+1)/3\rceil}\right\rfloor\\
        &=\left\lfloor\frac{2g-2+1}{\lceil (2\xi+1)/3\rceil}\right\rfloor. 
    \end{align*}
    
    (3-II-iii) Consider the remaining case of $b=1$ and $2> \lfloor 3(g-1-\xi/3)/\xi\rfloor$. Then $g-1$ is less than $\xi$, i.e., $g\le \xi$. In particular $\lfloor(2g-2+1)/\lceil (2\xi+1)/3\rceil\rfloor$ is at most $\lfloor(2\xi-2+1)/\lceil (2\xi+1)/3\rceil\rfloor<3$. If $\lfloor(2g-2+1)/\lceil (2\xi+1)/3\rceil\rfloor$ is equal to $2$, then $2g-1\ge 2(2\xi/3+1)$, i.e., $g\ge 2\xi/3+2$. Thus we have $\xi(\Sigma_{g-1-\xi/3, 2})=3g-\xi-4>\xi$, and the decomposition \cref{eq:Mod3Decomp-ii} guarantees $\mu(g, 1, \xi)\ge 2$. 

    (3-III) Finally, in the case of $b=0$, cutting $\Sigma$ along a non-separating simple closed curve yields $\Sigma_{g-1, 2}$. Therefore we can conclude that
    \begin{align*}
        \mu(g, 0, \xi)
        &\ge \mu(g-1, 2, \xi)\\
        &\ge \min\left\{\left\lfloor\frac{3(g-1)-2+2}{\xi+1}\right\rfloor,\ \left\lfloor\frac{2(g-1)-2+2}{\lceil (2\xi+1)/3\rceil}\right\rfloor\right\}\\
        &=\left\lfloor\frac{2g-2}{\lceil (2\xi+1)/3\rceil}\right\rfloor. 
    \end{align*}

    Combining the above results, we obtain the claim. 
\end{proof}

By setting $m(g,b,k)=\mu(g,b,3g-2+b+k)$, we have \Cref{theorem:B}.

\subsection{Relatively hyperbolic case}

In this subsection, we discuss the cases in which the $k$-multicurve graph $\multicurve{k}(\Sigma_{g,b})$ is relatively hyperbolic.

\begin{theorem}\label{theorem:relative-hyperbolicity}
    The $k$-multicurve graph $\multicurve{k}(\Sigma_{g,b})$ is relatively hyperbolic exactly for the values of $(g,b,k)$ shown in \Cref{table:list-for-case-of-rel-hyperbolic}.
    \begin{table}[htbp]
        \centering
        \begin{tabular}{c|c|c}
            \hline
            $g$ & $b$ & $k$ \\ \hline
            even & even ($\ge 2$) & $(3g+b)/2$ \\
            even & $0$ & $3g/2,\ (3g+2)/2$ \\
            odd & $0,\,2$ & $(3g+3)/2$ \\
            odd & odd ($\ge 3$) & $(3g+b)/2$ \\
            \hline 
            \end{tabular} \vspace{2mm}
            \caption{Values of $(g,b,k)$ for which $\multicurve{k}(\Sigma_{g,b})$ is relatively hyperbolic.}
            \label{table:list-for-case-of-rel-hyperbolic}
    \end{table}
\end{theorem}

First, we prepare two elementary lemmas that will be used later.
Note that the complexity $\xi(S)$ of an essential subsurface $S\subset \Sigma$ is equal to the number of curves in a pants decomposition for $S$.

\begin{lemma}\label{lemma:maximal-complexity-subsurface}
    If $X\subset \Sigma$ is an essential subsurface with $\xi(X)=\xi(\Sigma)$, then $X$ is isotopic to $\Sigma$.
\end{lemma}

\begin{lemma}\label{lemma:superadditivity-for-complexity}
    If $Y,Z \subset \Sigma$ are disjoint essential proper subsurfaces, then $\xi(\Sigma)\geq \xi(Y)+\xi(Z)+1$.
\end{lemma}

In the rest of this subsection, we prove \Cref{theorem:relative-hyperbolicity}.
By \Cref{proposition:witness-for-k-multicurve}, whether a subsurface is a witness for $\multicurve{k}(\Sigma_{g,b})$ is determined by its complexity.
Thus, by \Cref{theorem:classification-hyperbolic-relhyperbolic-thick}, for each $g,b$, we determine an integer $\xi$ with $1\leq \xi\leq 3g-3+b$ that satisfies the following conditions.
\begin{enumerate}[(A)]
    \item There exists a pair of disjoint, connected and essential subsurfaces whose complexities are at least $\xi$, and 
    \item Whenever co-connected, connected and essential subsurfaces $Y,Z\subset \Sigma$ are disjoint and their complexities are at least $\xi$, we have $Y^c=Z$ (up to isotopy).
\end{enumerate}

\begin{lemma}\label{lemma:sufficient-condition-for-A}
    If an integer $\xi$ with $1\leq \xi\leq 3g-3+b$ satisfies the condition (A), then $\xi$ is less than $(3g-3+b)/2$.
\end{lemma}

\begin{proof}
    We assume that $\xi$ with $\xi\geq (3g-3+b)/2$ satisfies the condition (A).
    Then, we can take a pair of disjoint essential subsurfaces $Y,Z\subset \Sigma$ whose complexities are at least $\xi$.
    From \Cref{lemma:superadditivity-for-complexity}, we have 
    \begin{align*}
        3g-3+b=\xi(\Sigma)\geq \xi(Y)+\xi(Z)+1\geq 3g-2+b.
    \end{align*}
    This is a contradiction.
    Thus, we obtain the conclusion.
\end{proof}

\begin{claim}
    Suppose that either (i) $g$ is even and $b\geq 2$ is even or (ii) $g$ is odd and $b\geq 3$ is odd.
    Then, the only integer $\xi$ satisfying both conditions (A) and (B) is $(3g-4+b)/2$.
\end{claim}

\begin{proof}
    We begin by considering assumption (i).

    (1) In the case of $\xi\geq (3g-2+b)/2$, the integer $\xi$ cannot satisfy the condition (A) by \Cref{lemma:sufficient-condition-for-A}.

    (2) We consider the case of $\xi=(3g-4+b)/2$.
    Then, we can take disjoint essential subsurfaces $Y,Z\subset \Sigma$ whose complexities are at least $(3g-4+b)/2$ (For instance, each of them has genus $g/2$ and $b/2+1$ boundary components).
    Therefore, the condition (A) is verified.
    Let $Y,Z$ be any pair of disjoint co-connected essential subsurfaces with $\xi(Y)$ and $\xi(Z)$ at least $(3g-4+b)/2$.
    If $\xi(Y)>(3g-4+b)/2$ or $\xi(Z)>(3g-4+b)/2$, \Cref{lemma:superadditivity-for-complexity} yields a contradiction.
    Therefore, we must have $\xi(Y)=\xi(Z)=(3g-4+b)/2$.
    By \Cref{lemma:superadditivity-for-complexity} again, we have 
    \begin{equation*}
        \xi(Y^c)+\xi(Y)+1\leq 3g-3+b.
    \end{equation*}
    Thus, $\xi(Y^c)\leq (3g-4+b)/2$.
    Since $Y$ and $Z$ are disjoint, we find $\xi(Z)\leq \xi(Y^c)$.
    Therefore, from \Cref{lemma:maximal-complexity-subsurface}, we obtain $Z=Y^c$.
    Thus, the condition (B) is also verified.

    (3) Finally, we consider the case of $\xi\leq (3g-6+b)/2$.
    Let $Y=\Sigma_{g/2,b/2+1}$ and $Z=\Sigma_{g/2,b/2}$.
    We glue one boundary component of $Y$ and one boundary component of $Z$ to two distinct boundary components of a pair of pants.
    The resulting surface is $\Sigma_{g,b}$.
    Therefore, there exist disjoint, co-connected essential subsurfaces $Y,Z\subset \Sigma$ satisfying that $\xi(Y)$ and $\xi(Z)$ are at least $\xi$, and that the union $Y\cup Z$ does not exhaust $\Sigma$. 
    Thus, the condition (B) cannot be satisfied.

    We now turn to statement (ii).
    As in the proof of statement (i), we divide the argument into three cases: (1) $\xi\geq(3g-2+b)/2$, (2) $\xi=(3g-4+b)/2$, or (3) $\xi \leq (3g-6+b)/2$.
    Then, for cases (1), (2), we can discuss in a similar way as statement (i).
    Note that, in case (2), the condition (A) is verified by $b\geq 3$.
    We consider the case of (3), that is, $\xi \leq (3g-6+b)/2$.
    For $Y=Z=\Sigma_{(g-1)/2,\,(b-1)/2}$, gluing them to a single pair of pants and a single annulus, 
    we obtain $\Sigma_{g,b}$ (see \Cref{figure:gluing-surface-odd-case}).
    Since $\xi(Y)=\xi(Z)\geq \xi$,
    The condition (B) cannot be satisfied.
    \begin{figure}
        \vspace{6mm}
       \centering
       \begin{overpic}[]{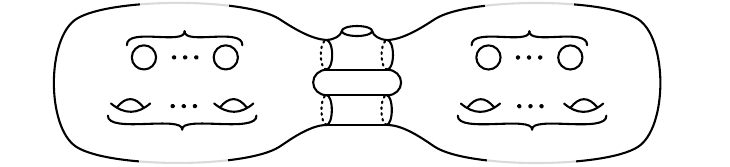}
           \put(12,17){$Y$}
           \put(83,17){$Z$}
           \put(21,22){$\dfrac{b-1}{2}$}
           \put(21,0){$\dfrac{g-1}{2}$}
           \put(68.5,22){$\dfrac{b-1}{2}$}
           \put(68.5,0){$\dfrac{g-1}{2}$}
       \end{overpic}
       \caption{}
       \label{figure:gluing-surface-odd-case}
    \end{figure} 
\end{proof}

The equivalence between the condition $\xi(X)\geq (3g-4+b)/2$  for an essential subsurface $X\subset \Sigma$ and the condition that $X$ is a witness for $\multicurve{k}(\Sigma)$ holds exactly when $k$ satisfies
\begin{equation*}
    3g-3+b-(k-1)=\frac{3g-4+b}{2} 
\end{equation*}
that is, $k=(3g+b)/2$.
This confirms the first and fourth rows of \Cref{table:list-for-case-of-rel-hyperbolic}.

\begin{claim}
    Let $g$ be odd and $b=1$.
    Then no integer $\xi$ with $1\leq \xi \leq 3g-3+b$ satisfies both conditions (A) and (B).
\end{claim}

\begin{proof}
    We first consider the case of $\xi\geq (3g-3)/2$.
    Suppose that we take essential disjoint subsurfaces $Y, Z \subset \Sigma$ of complexities at least $\xi$.
    If either $\xi(Y)>\xi$ or $\xi(Z)>\xi$, then we obtain a contradiction from \Cref{lemma:superadditivity-for-complexity}.
    Therefore we have $\xi(Y)=\xi(Z)=(3g-3)/2$.
    However, by \Cref{theorem:MaximumNumberForSubsurfaces}, 
    \begin{equation*}
        \mu(g,1,(3g-3)/2)\leq \left\lfloor \frac{2g-1}{\lceil (3g-2)/3 \rceil} \right\rfloor =\left\lfloor 2-\frac1{g}\right\rfloor\leq 1.
    \end{equation*}
    Thus the condition (A) does not hold.
\end{proof}

\begin{claim}
    Suppose that either (i) $g$ is even and $b$ is odd, or (ii) $g$ is odd and $b\geq 4$ is even.
    Then, no integer $\xi$ with $1\leq \xi\leq 3g-3+b$ satisfies both conditions (A) and (B).
\end{claim}

\begin{proof}
    If $\xi\geq (3g-3+b)/2$, in either case (i) or (ii), $\xi$ cannot satisfy the condition (A) by \Cref{lemma:sufficient-condition-for-A}.

    We consider the case of $\xi<(3g-3+b)/2$, that is, $\xi\leq (3g-5+b)/2$.
    We define surfaces $Y$ and $Z$ as follows:
    \begin{align*}
        &Y=\Sigma_{g/2,\,(b-1)/2}, & Z&=\Sigma_{g/2,\,(b-1)/2} &\ &\text{in case (i),}\\
        &Y=\Sigma_{(g+1)/2,\,b/2-2}, & Z&=\Sigma_{(g-1)/2,\,b/2+1} &\  &\text{in case (ii).}
    \end{align*}
    Then we glue one boundary component of $Y$ and one boundary component of $Z$ to two distinct boundary components of a pair of pants.
    The resulting surface is $\Sigma_{g,b}$.
    Therefore, from $\Sigma_{g,b}$, we can choose disjoint, co-connected essential subsurfaces $Y,Z\subset \Sigma$ such that $\xi(Y),\xi(Z)\geq \xi$ and the union $Y\cup Z$ does not exhaust $\Sigma$.
    This implies that the condition (B) cannot be satisfied.
\end{proof}

\begin{claim}
    Let $g$ be even and $b=0$.
    If an integer $\xi$ satisfies both conditions (A) and (B), then $\xi$ is equal to $(3g-4)/2$ or $(3g-6)/2$.
\end{claim}

\begin{proof}
    If $\xi\geq (3g-2)/2$, then $\xi$ cannot satisfy the condition (A) by \Cref{lemma:sufficient-condition-for-A}.
    
    Let $\xi=(3g-4)/2$.
    For $Y=Z=\Sigma_{g/2,1}$, gluing boundary components of $Y$ and $Z$ to each other, we obtain $\Sigma_{g,0}$.
    Since $\xi(Y)=\xi(Z)=(3g-4)/2$, the condition (A) is verified.
    Let $Y, Z\subset \Sigma$ be any pair of disjoint, co-connected essential subsurfaces with $\xi(Y)\geq (3g-4)/2$, $\xi(Z)\geq (3g-4)/2$.
    Then, since
    \begin{equation*}
        \xi(Y^c)\leq \xi(\Sigma)-\xi(Y)-1\leq \frac{3g-4}{2} \text{ and }Z\subset Y^c,
    \end{equation*}
    we have $\xi(Y^c)=\xi(Z)$.
    Thus we have $Y^c=Z$ up to isotopy.

    Let $\xi=(3g-6)/2$.
    We can verify the condition (A) in the same way in the case of $\xi=(3g-4)/2$.
    Let $Y, Z\subset \Sigma$ be any pair of disjoint, co-connected essential subsurfaces with $\xi(Y)\geq (3g-6)/2$, $\xi(Z)\geq (3g-6)/2$.
    Then, since
    \begin{equation*}
        \xi(Y^c)\leq \xi(\Sigma)-\xi(Y)-1\leq \frac{3g-2}{2} \text{ and } Z\subset Y^c,
    \end{equation*}
    we have $(3g-6)/2\leq \xi(Z)\leq \xi(Y^c)\leq (3g-2)/2$.
    Therefore we may assume that $(\xi(Y),\xi(Z))$ is equal to either 
    \begin{enumerate}[(i)]
        \item $\left(\frac{3g-6}{2}, \frac{3g-6}{2}\right)$, $\left(\frac{3g-6}{2}, \frac{3g-4}{2}\right),
        \left(\frac{3g-6}{2}, \frac{3g-2}{2}\right)$, \vspace{1mm}
        \item $\left(\frac{3g-4}{2}, \frac{3g-4}{2}\right)$, or \vspace{1mm}
        \item $\left(\frac{3g-4}{2}, \frac{3g-2}{2}\right)$.
    \end{enumerate}
    We consider case (i).
    Since $\xi(Y)$ is divisible by $3$, the subsurface $Y\subset \Sigma$ has at least $3$ boundary components.
    It follows that
    \begin{equation*}
        \xi(\Sigma) \geq \xi(Y) + \xi(Y^c) +3.
    \end{equation*}
    Therefore, $\xi(Y^c) \leq (3g-6)/2$.
    Since $Z\subset Y^c$, the complexity $\xi(Z)$ must equal $(3g-6)/2$ and hence we conclude $Z=Y^c$.


    Case (ii) reduces to the discussion for the case $\xi=(3g-4)/2$.
    Case (iii) contradicts \Cref{lemma:superadditivity-for-complexity}.

    Suppose that $\xi\leq (3g-8)/2$.
    Let $Y=Z=\Sigma_{(g-2)/2,\,2}$, and we can glue $Y$ and $Z$ to a single four-holed sphere along their boundaries so that the resulting surface is $\Sigma_{g,0}$.
    Thus, the condition (B) cannot be satisfied in this case.
\end{proof}

The equivalence between the condition $\xi(X)\geq (3g-4)/2$ (resp. $(3g-6)/2$) for an essential subsurface $X\subset \Sigma$ and the condition that $X$ is a witness for $\multicurve{k}(\Sigma)$ holds exactly when $k$ satisfies
\begin{equation*}
    3g-3-(k-1)=\frac{3g-4}{2},\ \left(\text{resp.}\ 3g-3-(k-1)=\frac{3g-6}{2}\right)
\end{equation*}
that is, when $k=3g/2$, (resp. $k=(3g+2)/2$).
This confirms the second row of \Cref{table:list-for-case-of-rel-hyperbolic}.

\begin{claim}\label{claim:g-odd-b-2}
    Let $g$ be odd and $b=2$. 
    Then, the only integer $\xi$ satisfying both conditions (A) and (B) is $(3g-3)/2$.
\end{claim}

\begin{proof}
    If $\xi\geq (3g-1)/2$, then $\xi$ cannot satisfy the condition (A) by \Cref{lemma:sufficient-condition-for-A}.

    Let $\xi=(3g-3)/2$.
    For $Y=Z=\Sigma_{(g-1)/2,\,3}$, gluing two components of $\partial Y$ and two components of $\partial Z$ to each other, we obtain $\Sigma_{g,2}$.
    Therefore the condition (A) is verified.
    Let $Y,Z\subset \Sigma$ be any pair of disjoint, co-connected essential subsurfaces with $\xi(Y)\geq (3g-3)/2,$\ $\xi(Z) \geq (3g-3)/2$.
    Then, by \Cref{lemma:maximal-complexity-subsurface}, we have 
    \begin{equation*}
        \xi(Y^c) \leq \xi(\Sigma)-\xi(Y)-1  \leq \frac{3g-1}{2}.
    \end{equation*}
    Now, since $(3g-3)/2\leq \xi(Z)\leq \xi(Y^c)\leq (3g-1)/2$.
    Therefore, we may assume that
    \begin{equation*}
        (\xi(Y),\xi(Z)) = \left(\frac{3g-3}{2}, \frac{3g-3}{2}\right)\ \text{or}\ \left(\frac{3g-3}{2}, \frac{3g-1}{2}\right).
    \end{equation*}
    Moreover, we may assume that at most one boundary component of $Y$ is homotopic to $\partial \Sigma$.
    Since $\xi(Y)$ is divisible by $3$, the number of boundary components of $Y$ is at least $3$. 
    Therefore,
    \begin{equation*}
        \xi(\Sigma)\geq \xi(Y)+\xi(Y^c)+2,
    \end{equation*}
    and hence we have $\xi(Y^c)\leq (3g-3)/2$.
    Thus, $(\xi(Y),\xi(Z)) = ((3g-3)/2,(3g-3)/2)$, and we have $Z=Y^c$.
    


    Suppose that $\xi\leq (3g-5)/2$.
    Let $Y=Z=\Sigma_{(g-1)/2,2}$, and we glue one boundary component of $Y$ and one boundary component of $Z$ to two distinct boundary components of a twice holed torus.
    The resulting surface is $\Sigma_{g,2}$.
    Now, since $\xi(Y)\geq \xi$ and $\xi(Z)\geq \xi$, the condition (B) cannot be satisfied in the case of $\xi\leq (3g-5)/2$.
\end{proof}

The equivalence between the condition $\xi(X)\geq (3g-3)/2$ for an essential subsurface $X\subset \Sigma$ and the condition that $X$ is a witness for $\multicurve{k}(\Sigma)$ holds exactly when $k$ satisfies
\begin{equation*}
    3g-3+b-(k-1)=\frac{3g-3}{2},
\end{equation*}
that is, $k=(3g+3)/2$.
This confirms the case of $b=2$ in the third row of \Cref{table:list-for-case-of-rel-hyperbolic}.

The following claim can be shown in essentially the same way as \Cref{claim:g-odd-b-2}.
Since there are a few minor differences, we include a proof for completeness.
\begin{claim}
    Let $g$ be odd and $b=0$. 
    Then, an integer $\xi$ satisfying both conditions (A) and (B) is $(3g-5)/2$.
\end{claim}

\begin{proof}
    If $\xi\geq (3g-3)/2$, then $\xi$ cannot satisfy the condition (A) by \Cref{lemma:sufficient-condition-for-A}.

    Let $\xi=(3g-5)/2$.
    For $Y=Z=\Sigma_{(g-1)/2,\,2}$, gluing $Y$ and $Z$ along their boundary components, we obtain $\Sigma_{g,0}$.
    Since $\xi(Y)$ and $\xi(Z)$ are equal to $(3g-5)/2$, the condition (A) is verified.
    Let $Y,Z\subset \Sigma$ be any pair of disjoint, co-connected essential subsurfaces with $\xi(Y)\geq (3g-5)/2$, $\xi(Z)\geq (3g-5)/2$.
    Then, by \Cref{lemma:superadditivity-for-complexity}, we have 
    \begin{equation*}
        \xi(Y^c)\leq \xi(\Sigma)-\xi(Y)-1\leq \frac{3g-3}{2}.
    \end{equation*}
    Therefore we may assume 
    \begin{equation*}
        (\xi(Y),\xi(Z))=\left(\frac{3g-5}{2},\frac{3g-5}{2}\right) \text{or} \left(\frac{3g-5}{2},\frac{3g-3}{2}\right).
    \end{equation*}
    Then, since $\xi(Y)\equiv 2 \pmod 3$, the number of boundary components of $Y$ is at least $2$.
    Therefore, we have 
    \begin{equation*}
        \xi(\Sigma) \geq \xi(Y)+\xi(Y^c)+2.
    \end{equation*}
    Hence, we obtain $\xi(Y^c)\leq (3g-5)/2$.
    By $Z\subset Y^c$ and $\xi(Z)\leq \xi(Y^c)$, we have $(\xi(Y),\xi(Z))=((3g-5)/2,(3g-5)/2)$ and $Y^c=Z$.
    Thus the condition (B) is verified.


    Suppose that $\xi\leq (3g-7)/2$.
    Let $Y=Z=\Sigma_{(g-1)/2,\,1}$, and we glue one boundary component of $Y$ and one boundary component of $Z$ to two distinct boundary components of a twice-holed torus.
    The resulting surface is $\Sigma_{g,0}$.
    Since $\xi(Y)\geq \xi$ and $\xi(Z)\geq \xi$, the condition (B) cannot be satisfied in the case of $\xi\leq (3g-7)/2$. 
\end{proof}

The equivalence between the condition $\xi(X)\geq (3g-5)/2$ for an essential subsurface $X\subset \Sigma$ and the condition that $X$ is a witness for $\multicurve{k}(\Sigma)$ holds exactly when $k$ satisfies
\begin{equation*}
    3g-3-(k-1)=\frac{3g-5}{2},
\end{equation*}
that is, when $k=(3g+3)/2$,.
This confirms the case of $b=0$ in the third row of \Cref{table:list-for-case-of-rel-hyperbolic}.

\section{Quasi-isometric relations between $k$-multicurve graphs}

Finally, we discuss some natural questions arising from the results of this paper.
For the $k$-multicurve graph $\multicurve{k}(\Sigma_{g,b})$, the quasi-flat rank $m(g,b,k)$ is a quasi-isometry invariant.
Fixing $g$ and $b$, the value of $m(g,b,k)$ may coincide for distinct values $k$.
In particular, there are many values of $k$ such that $\multicurve{k}(\Sigma_{g,b})$ is Gromov hyperbolic.
This leads to the following natural question: \textit{if $m(g,b,k)=m(g,b,k')$ for distinct $k, k'$, then are the graphs $\multicurve{k}(\Sigma_{g,b})$ and $\multicurve{k'}(\Sigma_{g,b})$ quasi-isometric? }
From recent work by Aramayona, Parlier and Webb \cite{aramayona2025hyperbolic}, we obtain the following corollary.

\begin{corollary}
    For every $k$ with $2\leq k\leq 3g-3+b$, the $k$-multicurve graph $\multicurve{k}(\Sigma)$ is not quasi-isometric to the curve graph $\mathcal{C}(\Sigma)$.
\end{corollary}

For triples $(g,b,k)$ satisfying $m(g,b,k)=2$, the $k$-multicurve graph $\multicurve{k}(\Sigma_{g,b})$ can be relatively hyperbolic in some cases and thick in others, as shown in \Cref{table:list-for-case-of-rel-hyperbolic}, and these cases are not quasi-isometric to each other.

\bibliography{multicurve-interpolating.bib}
\bibliographystyle{alpha}

\end{document}